\documentclass[12pt,reqno]{amsart}
\usepackage{mathrsfs}
\usepackage{amssymb}
\usepackage{txfonts}
\usepackage{amsfonts}
\usepackage{amsmath}
\usepackage{graphicx}
\usepackage{hyperref}
\usepackage{float}
\usepackage{epstopdf}
\usepackage{color}

\makeatletter \@namedef{subjclassname@2000}{\emph{2020 Mathematics
Subject Classification}} \makeatother

\allowdisplaybreaks
\newtheorem{theorem}{Theorem}[section]

 \newtheorem{lemma}[theorem]{Lemma}
 \newtheorem{proposition}[theorem]{Proposition}

\numberwithin{equation}{section}

\renewcommand{\theequation}{\thesection.\arabic{equation}}

\input{tcilatex}

\def\oint{\fint}

\def\FRAME#1#2#3#4#5#6#7#8
{
 \begin{figure}[ptbh]
 \begin{center}
 \includegraphics[height=#3]{#7}
 \caption{#5}
 \label{#6}
 \end{center}
 \end{figure}
}

\begin{document}

\title{Anisotropic versions of the Brezis-Van
Schaftingen-Yung approach at $s=1$ and $s=0$}

\pagestyle{plain}

\author[Q. Gu]{Qingsong Gu}
\address[Q. Gu]{Department of Mathematics, Nanjing University, Nanjing 210093, China} \email{\href{mailto: Qingsong Gu
<001gqs@163.com>}{001gqs@163.com} }

\author[Q. Huang]{Qingzhong Huang}
\address[Q. Huang]{College of Data Science, Jiaxing University, Jiaxing, 314001, China} \email{\href{mailto: Qingzhong Huang
<hqz376560571@163.com>}{hqz376560571@163.com} }

\thanks {The second author was supported
by NSFC (No. 11701219).}

\maketitle
\begin{abstract}
In 2014, Ludwig showed the limiting behavior of the anisotropic
Gagliardo $s$-seminorm of a function $f$ as $s\rightarrow 1^-$ and
$s\rightarrow0^+$, which extend the results due to
Bourgain-Brezis-Mironescu(BBM) and Maz'ya-Shaposhnikova(MS)
respectively. Recently, Brezis, Van Schaftingen and Yung
provided a different approach by replacing the strong $L^p$ norm in
the Gagliardo $s$-seminorm by the weak $L^p$ quasinorm. They
characterized the case for $s=1$ that complements the BBM formula.
The corresponding MS formula for $s=0$ was later established by Yung
and the first author. In this paper, we follow the approach of
Brezis-Van Schaftingen-Yung and show the anisotropic versions of
$s=1$ and $s=0$. Our result generalizes the work by Brezis, Van
Schaftingen, Yung and the first author and complements the work by
Ludwig.

\end{abstract}

\bigskip

\section{\large  Introduction}

For $1\leq p<\infty$ and $0<s<1$, the Gagliardo $s$-seminorm of a
function $f\in L^p(\mathbb R^n)$ is defined as
\begin{equation}\label{eq1}
||f||^p_{W^{s,p}(\mathbb R^n)}:=\int_{\mathbb R^n}\int_{\mathbb
R^n}\frac{|f(x)-f(y)|^p}{|x-y|^{n+sp}}dxdy,
\end{equation}
where $|\cdot|$ denotes the Euclidean norm on $\mathbb R^n$. This
seminorm arises in connection with many problems in the theory of
partial differential equations (e.g. \cite{BB02,DPV,Maz}).

The limiting behavior of $||f||^p_{W^{s,p}(\mathbb R^n)}$ for
$s\rightarrow1^-$ was firstly studied by Bourgain, Brezis, and
Mironescu \cite{BB01}, in which they obtained the following BBM
formula: for $1\leq p<\infty$ and $f\in W^{1,p}(\mathbb R^n)$,
\begin{equation}\label{eq2}
\lim\limits_{s\rightarrow1^-}(1-s)||f||^p_{W^{s,p}(\mathbb
R^n)}=\frac{1}{p}k(p,n)||\nabla f||^p_{L^p(\mathbb R^n)},
\end{equation}
where
\begin{equation}\label{k}
k(p,n)=\int_{S^{n-1}}|e\cdot
\omega|^pd\omega=\frac{2\Gamma((p+1)/2)\pi^{(n-1)/2}}{\Gamma((n+p)/2)}.
\end{equation}
Here $e\in S^{n-1}$ is any fixed unit vector and $e\cdot \omega$ is the
inner product of $e,\omega\in S^{n-1}$.

The limiting behavior of $||f||^p_{W^{s,p}(\mathbb R^n)}$ for
$s\rightarrow0^+$ was considered by Maz'ya and Shaposhnikova
\cite{MS}, in which they obtained the following MS formula: if $f\in
W^{s,p}(\mathbb R^n)$ for all $s\in(0,1)$, then
\begin{equation}\label{eq2'}
\lim\limits_{s\rightarrow0^+}s||f||^p_{W^{s,p}(\mathbb
R^n)}=\frac{2n}{p}|B^n|||f||^p_{L^p(\mathbb R^n)},
\end{equation}
where $|B^n|$ is the volume of the Euclidean unit ball $B^n$ in
$\mathbb R^n$.

\medskip
We say that a set $K\subset\mathbb R^n$ is a convex body if it is
compact and convex and has non-empty interior. For an
origin-symmetric convex body $K\subset\mathbb R^n$, the Minkowski
functional $||\cdot||_K$:
\begin{equation}\label{Min}||x||_K=\inf\{\lambda\geq 0:\  x\in\lambda K\}\end{equation}
defines a norm on $\mathbb R^n$ for all $x\in \mathbb R^n$.
Moreover, we use $Z_{p}^*K$ to denote the polar $L^p$ moment body of
$K$, whose norm $||\cdot||_{Z_{p}^*K}^p$ is given by
\begin{equation}\label{moment}
||z||_{Z_{p}^*K}^p=\frac{n+p}{2}\int_K|z\cdot y|^pdy,\qquad\qquad z\in\mathbb R^n.
\end{equation}

In recent years, the research on anisotropic Sobolev spaces have
received considerable attention (see e.g.
\cite{AF,BC,BCS,BCS18,CFR,CNV,FMP,LMP,Ma,NS,XY}). In particular, the
anisotropic Gagliardo $s$-seminorm is obtained by replacing the
Euclidean norm $|x-y|$ by $||x-y||_K$ in $\eqref{eq1}$. The limiting
behavior of the anisotropic Gagliardo $s$-seminorm for
$s\rightarrow1^-$ and $s\rightarrow0^+$ were established by Ludwig
\cite{Lud,Lud2}. More specifically, Ludwig proved the following two
formulas:

\noindent{\bf The anisotropic version of BBM forumla}. If $1\leq
p<\infty$ and $f\in W^{1,p}(\mathbb R^n)$ has compact support, then
\begin{equation}\label{eq3}
\lim\limits_{s\rightarrow1^-}(1-s)\int_{\mathbb R^n}\int_{\mathbb
R^n}\frac{|f(x)-f(y)|^p}{||x-y||^{n+sp}_K}dxdy=\frac{2}{p}\int_{\mathbb
R^n}||\nabla f(x)||_{Z_p^*K}^pdx.
\end{equation}

\noindent{\bf The anisotropic version of MS forumla}. If $1\leq
p<\infty$, $f\in W^{s,p}(\mathbb R^n)$ for all $s\in(0,1)$ and $f$
has compact support, then
\begin{equation}\label{eq3'}
\lim\limits_{s\rightarrow0^+}s\int_{\mathbb R^n}\int_{\mathbb
R^n}\frac{|f(x)-f(y)|^p}{||x-y||^{n+sp}_K}dxdy=\frac{2n}{p}|K|||
f||_{L^p(\mathbb R^n)}^p.
\end{equation}
\noindent From the fact that $||z||_{Z_{p}^*B^n}^p=k(p,n)|z|^p/2$,
it is easy to see that \eqref{eq3} recovers \eqref{eq2} and
\eqref{eq3'} recovers \eqref{eq2'} if $K=B^n$. We shall mention the
$L_p$ moment body $Z_pK$, the polar body of $Z_p^*K$, and the $L_p$
centroid body $\frac{2}{(n+p)|K|}Z_pK$ introduced by Lutwak and
Zhang \cite{LZ} have become important tools used in convex geometry,
probability theory, and the local theory of Banach spaces (see e.g.
\cite{FGP,HS,KM,Lud03,Lud05,LYZ,LYZ00,LYZ02,LYZ04,LYZ10,Pa,PW,P,YY}).

\bigskip

For $1\leq p<\infty$, the Marcinkiewicz (i.e. weak $L^p$) quasinorm
$[\cdot]_{M^p(\mathbb R^n\times\mathbb R^n,\mathcal L^{2n})}$ is
defined by
\begin{equation}\label{Mar}
[g]^p_{M^p(\mathbb R^n\times\mathbb R^n,\mathcal
L^{2n})}:=\sup_{\lambda>0} \lambda^p\mathcal L^{2n}\left(\{x\in
\mathbb R^n\times\mathbb R^n:\ |g(x)|\ge\lambda\}\right),
\end{equation}
$\mathcal L^{2n}$ denotes the Lebesgue measure on $\mathbb R^n\times
\mathbb R^n$. The Marcinkiewicz space $M^p(\mathbb R^n\times\mathbb
R^n,\mathcal L^{2n})$ modeled on $L^p$ contains all the functions
$g$ with $[g]_{M^p(\mathbb R^n\times\mathbb R^n,\mathcal
L^{2n})}<\infty$ (e.g. \cite{CR,G}).

Recently, Brezis, Van Schaftingen and Yung \cite{BSY} provided
a different approach by replacing the strong $L^p$ norm in the
Gagliardo $s$-seminorm by the weak $L^p$ quasinorm, which
complements the BBM formula. Precisely, they proved that there exist
a positive constant $c=c(n)$ such that for all $f\in
C_c^\infty(\mathbb R^n)$ and $1\leq p<\infty$,
\begin{equation}\label{eq5}
\frac{1}{n}k(p,n)||\nabla f||^p_{L^p(\mathbb
R^n)}\leq\left[\frac{f(x)-f(y)}{|x-y|^{\frac
np+1}}\right]^p_{M^p(\mathbb R^n\times\mathbb R^n,\mathcal L^{2n})}
\leq c||\nabla f||^p_{L^p(\mathbb R^n)},
\end{equation}
where $k(p,n)$ is defined in \eqref{k}.  Moreover, if we denote
\begin{equation*}
E_\lambda=\left\{(x,y)\in\mathbb R^n\times\mathbb R^n:\ x\neq
y,\frac{|f(x)-f(y)|}{|x-y|^{\frac{n}{p}+1}}\geq\lambda\right\},
\end{equation*}
then
\begin{equation}\label{eq7}
\lim_{\lambda\rightarrow\infty}\lambda^p\mathcal
L^{2n}(E_{\lambda})=\frac{1}{n}k(p,n)||\nabla f||^p_{L^p(\mathbb
R^n)}.
\end{equation}
Clearly, the first inequality in \eqref{eq5} is a direct consequence
of \eqref{eq7}.

Inspired by this idea, Yung and the first author \cite{GY}
established the corresponding MS formula for $s=0$, i.e. they
showed that for all $1\leq p<\infty$ and all $f\in L^p(\mathbb
R^n)$,
\begin{equation}\label{eq5'}
2|B^n|||f||^p_{L^p(\mathbb
R^n)}\leq\left[\frac{f(x)-f(y)}{|x-y|^{\frac
np}}\right]^p_{M^p(\mathbb R^n\times\mathbb R^n,\mathcal L^{2n})}
\leq 2^{p+1}|B^n|||f||^p_{L^p(\mathbb R^n)}.
\end{equation}
Moreover, if we denote
\begin{equation*}
\widetilde E_\lambda=\left\{(x,y)\in\mathbb R^n\times\mathbb R^n:\ x\neq
y,\frac{|f(x)-f(y)|}{|x-y|^{\frac{n}{p}}}\geq\lambda\right\},
\end{equation*}
then
\begin{equation}\label{eq7'}
\lim_{\lambda\rightarrow0^+}\lambda^p\mathcal L^{2n}(\widetilde
E_{\lambda})=2|B^n|||f||^p_{L^p(\mathbb R^n)}.
\end{equation}
Similarly, the first inequality in \eqref{eq5'} can be deduced from
\eqref{eq7'}. The approach of Brezis-Van Schaftingen-Yung also
inspires the work \cite{BSY2,Pol}.

Motivated by the two formulas of Ludwig (\eqref{eq3} and \eqref{eq3'}), it is natural to
ask what happens if we replace the $|x-y|$ in \eqref{eq5} and
\eqref{eq5'} by $||x-y||_K$?

In this paper, we will give a positive
answer to this question by the following theorem, which establishes
anisotropic versions of formulas \eqref{eq5} and \eqref{eq5'} and
their limiting behaviors \eqref{eq7} and \eqref{eq7'}.
\begin{theorem}\label{mainthm}
Let $1\leq p<\infty$ and let $K$ be an origin-symmetric convex body
in $\mathbb R^n$. Then

\medskip

\noindent \emph{(a)} there exist a positive constant $C=C(n)$ such
that for all $f\in C_c^\infty(\mathbb R^n)$,
\begin{equation}\label{eqqe5'}
\frac{2}{n}\int_{\mathbb R^n}||\nabla
f(x)||_{Z_p^*K}^pdx\leq\left[\frac{f(x)-f(y)}{||x-y||_K^{\frac
np+1}}\right]^p_{M^p(\mathbb R^n\times\mathbb R^n,\mathcal L^{2n})}
\leq C\int_{\mathbb R^n}||\nabla f(x)||_{Z_p^*K}^pdx.
\end{equation}
Moreover, if we denote
\begin{equation*}
E_{\lambda,K}=\left\{(x,y)\in\mathbb R^n\times\mathbb R^n:\ x\neq
y,\frac{|f(x)-f(y)|}{||x-y||_K^{\frac{n}{p}+1}}\geq\lambda\right\},
\end{equation*}
then
\begin{equation}\label{eqqe7'}
\lim_{\lambda\rightarrow\infty}\lambda^p\mathcal
L^{2n}(E_{\lambda,K})=\frac{2}{n}\int_{\mathbb R^n}||\nabla
f||_{Z_p^*K}^pdx.
\end{equation}

\bigskip

\noindent \emph{(b)} for all $f\in L^p(\mathbb R^n)$,
\begin{equation}\label{eqeq5'}
2|K|||f||^p_{L^p(\mathbb
R^n)}\leq\left[\frac{f(x)-f(y)}{||x-y||_K^{\frac
np}}\right]^p_{M^p(\mathbb R^n\times\mathbb R^n,\mathcal L^{2n})}
\leq 2^{p+1}|K|||f||^p_{L^p(\mathbb R^n)}.
\end{equation}
Moreover, if we denote
\begin{equation*}
\widetilde E_{\lambda,K}=\left\{(x,y)\in\mathbb R^n\times\mathbb R^n:\ x\neq
y,\frac{|f(x)-f(y)|}{||x-y||_K^{\frac{n}{p}}}\geq\lambda\right\},
\end{equation*}
then
\begin{equation}\label{eqeq7'}
\lim_{\lambda\rightarrow0^+}\lambda^p\mathcal
L^{2n}(\widetilde E_{\lambda,K})=2|K|||f||^p_{L^p(\mathbb
R^n)}.
\end{equation}

\end{theorem}

Also, the first inequalities in \eqref{eqqe5'} and \eqref{eqeq5'}
can be deduced from \eqref{eqqe7'} and \eqref{eqeq7'}, respectively.
Furthermore, we shall mention that by letting $K=B^n$, the fact that
$||z||_{Z_{p}^*B^n}^p=k(p,n)|z|^p/2$ yields that formula
\eqref{eqqe5'} recovers formula \eqref{eq5} (the second inequality
in \eqref{eqqe5'} differs from \eqref{eq5} up to the constant
$k(p,n)$). Similarly formula \eqref{eqeq5'} recovers formula
\eqref{eq5'} if $K=B^n$.

\section{\large Proof of Theorem \ref{mainthm}}\label{sec2}
Throughout the paper, we always assume $p\in[1,\infty)$ and $K$
denotes an origin-symmetric convex body in $\mathbb R^n$.
\subsection*{Proof of Theorem \ref{mainthm}(a)}
We first give some notations. For $\omega\in S^{n-1}$, let
$\omega^{\perp}=\{x\in \mathbb R^n: x\cdot \omega=0 \}$ and
$L_{\omega}=\{s \omega:s\in \mathbb R\}$. Let $\text{Aff}(n,1)$ be
the Grassmannian of lines in $\mathbb R^n$. Then for any $L\in
\text{Aff}(n,1)$, we can write $L=\hat{x}+L_{\omega}$ for some
$\omega\in S^{n-1}$ and $\hat{x}\in \omega^{\perp}$. Furthermore, we
may write $x=\hat{x}+s_x\omega$ for any point $x\in L$. The
following Blaschke-Petkantschin formula is from integral geometry (see
e.g. \cite[Theorem 7.2.7]{SW}).
\begin{equation}\label{BP}
\int_{\mathbb R^n}\int_{\mathbb R^n}g(x,y)d\mathcal H^n(x)d\mathcal
H^n(y)=\int_{\text{Aff}(n,1)}\int_L\int_Lg(x,y)|x-y|^{n-1}d\mathcal
H^1(x)d\mathcal H^1(y)dL,
\end{equation}
where $\mathcal H^k$ denotes the $k$-dimensional Hausdorff measure
on $\mathbb R^n$ and $dL$ denotes a suitably normalized rigid motion
invariant Haar measure on $\text{Aff}(n,1)$. Moreover, it follows
from \cite[Theorem 13.2.12]{SW} that for every measurable function
$h:\text{Aff}(n,1)\rightarrow[0,\infty)$,
\begin{equation}\label{atos}
\int_{\text{Aff}(n,1)}h(L)dL=\frac{1}{2}\int_{S^{n-1}}\int_{\omega^{\perp}}h(\hat{x}+L_{\omega})d\mathcal
H^{n-1}(\hat{x})d\mathcal H^{n-1}(\omega).
\end{equation}

The following remarkable Proposition for the one-dimensional case
established by Brezis, Van Schaftingen, and Yung \cite[Proposition
2.1]{BSY} plays a central role in the proof of formula \eqref{eq5}.
\begin{proposition}\label{BSYProp2.1}
There exists a universal constant $C$ such that for all $\gamma>0$
and $F\in C_c(\mathbb R)$ that
\begin{equation*}
\iint_{E(F,\gamma)}|x-y|^{\gamma-1}d\mathcal H^1(x)d\mathcal
H^1(y)\leq C\frac{5^{\gamma}}{\gamma}||F||_{L^1(\mathbb R)},
\end{equation*}
where \begin{equation}\label{EFL}E(F,\gamma):=\left\{(x,y)\in\mathbb
R\times\mathbb R:\ x\neq y,
\left|\int_y^xF\right|\geq|x-y|^{\gamma+1}\right\}.\end{equation}
\end{proposition}

Now we apply the Blaschke-Petkantschin formula to the above
Proposition and obtain an inequality generalizing \cite[Proposition 2.2]{BSY} (see the following Proposition \ref{BSYProp2.2}). It will be used in the proof of Theorem \ref{mainthm}(a).

For this purpose, we let $F$ be an operator
\begin{equation*}
F:\ \text{Aff}(n,1)\rightarrow C_c(\mathbb R),
\end{equation*}
 i.e., $F$ maps a line $L$ to a compactly supported continuous function $F_L$ defined on the line
 $L$.
Note that for any $L\in \text{Aff}(n,1)$, we can write
$L=\hat{x}+L_{\omega}$ for some $\omega\in S^{n-1}$ and $\hat{x}\in
\omega^{\perp}$. Thus, we may use $F_L(\hat{x}+s\omega)$ to denote
the function of $s\in\mathbb R$.

\begin{proposition}\label{BSYProp2.2}
For any positive integer $n$, there exists a constant $C=C(n)$ such
that for any operator $F:\ \emph{Aff}(n,1)\rightarrow C_c(\mathbb
R)$,
\begin{equation*}\label{dd}
\mathcal L^{2n}(E(F))\leq
C\int_{S^{n-1}}\int_{\omega^{\perp}}\int_{\mathbb
R}\left|F_{\hat{x}+L_{\omega}}(\hat{x}+s\omega)\right|d\mathcal
H^1(s)d\mathcal H^{n-1}(\hat{x})d\mathcal H^{n-1}(\omega),
\end{equation*}
where
$$E(F)=\left\{(x,y)\in \mathbb R^n\times\mathbb R^n:\ x\neq
y,\left|\int_{s_y}^{s_x}F_{L}(\hat{x}+s\omega)ds\right|\geq|x-y|^{n+1},
\text{$L$~passing~through~$x$ and $y$}\right\}.$$

\end{proposition}

\begin{proof} Denote by $\textbf{1}_{E(F)}$ the indicator function of the set
$E(F)$. The Blaschke-Petkantschin formula \eqref{BP} with
$g(x,y)=\textbf{1}_{E(F)}$ yields that
\begin{equation}\label{eqB-P1}
\mathcal L^{2n}(E(F))=\int_{\mathbb R^n}\int_{\mathbb
R^n}\textbf{1}_{E(F)}d\mathcal H^n(x)d\mathcal
H^n(y)=\int_{\text{Aff}(n,1)}\int_L\int_L\textbf{1}_{E(F)}\cdot|x-y|^{n-1}d\mathcal
H^1(x)d\mathcal H^1(y)dL.
\end{equation}
Denote by $\widetilde{F}_L(s)=F_L(\hat{x}+s\omega)$. Then $E(F)$ restricted on $L$ is just
 $E(\widetilde{F}_L,n)$ defined in \eqref{EFL}. Applying
Proposition \ref{BSYProp2.1} to  $\widetilde{F}_L$, we get
\begin{align}
&\int_L\int_L\textbf{1}_{E(F)}\cdot|x-y|^{n-1}d\mathcal H^1(x)d\mathcal H^1(y)\notag\\
=&\iint_{E(\widetilde{F}_L,n)}|s_x-s_y|^{n-1}d\mathcal H^1(s_x)d\mathcal H^1(s_y)\notag\\
\leq& C\frac{5^n}{n}\int_{\mathbb
R}\left|\widetilde{F}_L(s)\right|d\mathcal
H^1(s)=C\frac{5^n}{n}\int_{\mathbb
R}\left|F_L(\hat{x}+s\omega)\right|d\mathcal H^1(s).\label{eqB-P2}
\end{align}
Substituting \eqref{eqB-P2} into \eqref{eqB-P1}, together with
\eqref{atos}, we obtain
\begin{align*}
\mathcal L^{2n}(E(F))\leq
&C\frac{5^n}{n}\int_{\text{Aff}(n,1)}\int_{\mathbb
R}\left|F_L(\hat{x}+s\omega)\right|d\mathcal H^1(s)dL\notag\\
=&C\frac{5^n}{2n}\int_{S^{n-1}}\int_{\omega^{\bot}}\int_{\mathbb
R}\left|F_{\hat{x}+L_{\omega}}(\hat{x}+s\omega)\right|d\mathcal
H^1(s)d\mathcal H^{n-1}(\hat{x})d\mathcal H^{n-1}(\omega),
\end{align*}
as desired.
\end{proof}

Inspired by the technique developed in \cite{BSY}, the proof of
Theorem \ref{mainthm}(a) can be divided by the following two lemmas.
The second inequality of \eqref{eqqe5'} follows from Lemma \ref{L1}.
The limiting behavior \eqref{eqqe7'} will be established in Lemma
\ref{L2}.

\begin{lemma}\label{L1}
There exists $C=C(n)>0$ such that for all
$f\in C_c^\infty(\mathbb R^n)$,
\begin{equation}\label{eqrig}
\left[\frac{f(x)-f(y)}{||x-y||_K^{\frac np+1}}\right]^p_{M^p(\mathbb
R^n\times\mathbb R^n,\mathcal L^{2n})} \leq C\int_{\mathbb
R^n}||\nabla f||_{Z_p^*K}^pdx.
\end{equation}
\end{lemma}
\begin{proof} On one hand, for any $L\in \text{Aff}(n,1)$, we can write
$L=\hat{x}+L_{\omega}$ for some $\omega\in S^{n-1}$ and $\hat{x}\in
\omega^{\perp}$. Define 
\begin{equation*}
F_{L}(\hat{x}+s\omega):=\frac{\left|\nabla f(\hat{x}+s\omega)\cdot
\omega\right|^p}{\lambda^p||\omega||_K^{n+p}},\qquad\qquad s\in
\mathbb R.
\end{equation*}
Applying Proposition \ref{BSYProp2.2} to the above $F_L$, we see
that
\begin{equation*}E(F)=\left\{(x,y)\in \mathbb R^n\times\mathbb R^n:\ x\neq
y,\left|\int_{s_y}^{s_x}\frac{\left|\nabla f(\hat{x}+s\omega)\cdot
\omega\right|^p}{\lambda^p||\omega||_K^{n+p}}ds\right|\geq|x-y|^{n+1}\right\}.
\end{equation*}
Together with Fubini's theorem, the polar coordinate and
\eqref{moment}, we obtain
\begin{align}\label{r1}
\mathcal L^{2n}(E(F))&\leq
C\int_{S^{n-1}}\int_{\omega^{\perp}}\int_{\mathbb
R}\frac{\left|\nabla f(\hat{x}+s\omega)\cdot
\omega\right|^p}{\lambda^p||\omega||_K^{n+p}}d\mathcal H^{1}(s)d\mathcal H^{n-1}(\hat{x})d\mathcal H^{n-1}(\omega)\notag\\
&=\frac{C}{\lambda^p}\int_{S^{n-1}}\int_{\mathbb R^n}\left|\nabla
f(x)\cdot
\omega\right|^p||\omega||_K^{-n-p}dxd\mathcal H^{n-1}(\omega)\notag\\
&=\frac{C}{\lambda^p}\int_{\mathbb
R^n}\left(\int_{S^{n-1}}\left|\nabla f(x)\cdot
\omega\right|^p||\omega||_K^{-n-p}d\mathcal H^{n-1}(\omega)\right)dx\notag\\
&=\frac{2C}{\lambda^p}\int_{\mathbb R^n}||\nabla f(x)
||_{Z_p^*K}^pdx.
\end{align}

On the other hand, let $L=\hat{x}+L_{\omega}$ be the line passing
through $x,y\in \mathbb R^n$. Let $\tilde{f}(s)=f(\hat{x}+s\omega)$,
then $\tilde{f}^{\prime}(s)=\nabla f(\hat{x}+s\omega)\cdot \omega$.
From H\"{o}lder's inequality and the fact that $|x-y|=|s_x-s_y|$, we
have
\begin{align}\label{eq11}
|f(x)-f(y)|&=\left|\tilde{f}(s_x)-\tilde{f}(s_y)\right|\notag\\
&=\left|\int_{s_y}^{s_x}\tilde{f}^{\prime}(s)ds\right|\notag\\
&\leq \left|\int_{s_y}^{s_x}\left|\nabla f(\hat{x}+s\omega)\cdot \omega\right|ds\right|\notag\\
&\leq |x-y|^{\frac{p-1}{p}} \left|\int_{s_y}^{s_x}\left|\nabla
f(\hat{x}+s\omega)\cdot \omega\right|^pds\right|^{\frac1p}.
\end{align}
Note that $||x-y||_K=|x-y|\cdot ||\omega||_K$. Thus, \eqref{eq11}
implies that
\begin{align}
&\left\{(x,y)\in\mathbb R^n\times\mathbb R^n:\ x\neq y,\frac{|f(x)-f(y)|}{||x-y||_K^{\frac{n}{p}+1}}\geq\lambda\right\}\subseteq \notag\\
&\left\{ (x,y)\in\mathbb R^n\times\mathbb R^n:\ x\neq
y,\left|\int_{s_y}^{s_x}\frac{\left|\nabla f(\hat{x}+s\omega)\cdot
\omega\right|^p}{\lambda^p||\omega||_K^{n+p}}ds\right|\geq
|x-y|^{n+1}\right\}.\label{eq12}
\end{align}
Combining \eqref{r1} and \eqref{eq12}, we get
\begin{equation}\label{eq20}
\lambda^p\mathcal L^{2n}\left(\left\{(x,y)\in\mathbb
R^n\times\mathbb R^n:\ x\neq
y,\frac{|f(x)-f(y)|}{||x-y||_K^{\frac{n}{p}+1}}\geq\lambda\right\}\right)\leq
2C\int_{\mathbb R^N}||\nabla f(x)||_{Z_p^*K}^pdx.
\end{equation}
Hence by \eqref{Mar}, the desired inequality \eqref{eqrig} follows by taking the supremum of \eqref{eq20} for all
$\lambda>0$.
\end{proof}

\begin{lemma}\label{L2}
For $f\in C_c^\infty(\mathbb R^n)$, if we denote
\begin{equation}\label{eq6''}
E_{\lambda,K}=\left\{(x,y)\in\mathbb R^n\times\mathbb R^n:\ x\neq
y,\frac{|f(x)-f(y)|}{||x-y||_K^{\frac{n}{p}+1}}\geq\lambda\right\},
\end{equation}
then
\begin{equation}\label{eq7''}
\lim_{\lambda\rightarrow\infty}\lambda^p\mathcal
L^{2n}(E_{\lambda,K})=\frac{2}{n}\int_{\mathbb R^n}||\nabla
f(x)||_{Z_p^*K}^pdx.
\end{equation}

\end{lemma}

\begin{proof}
For $f\in C_c^\infty(\mathbb R^n)$, denote by $a:=||\nabla
f||_{L^\infty(\mathbb R^n)}$ and $b:=||\nabla^2
f||_{L^\infty(\mathbb R^n)}$. Obviously,
\begin{equation}\label{eq21}
|f(x)-f(y)|\leq a|x-y|,\qquad \text{for~any}~  x,y\in\mathbb R^n,
\end{equation}
and
\begin{equation}\label{eq22}
|f(x)-f(y)-\nabla f(x)\cdot(x-y)|\leq b|x-y|^2, \qquad
\text{for~any}~ x,y\in\mathbb R^n.
\end{equation}

For any $x\in \mathbb R^n$, $\omega\in S^{n-1}$ such that
$\nabla f(x)\cdot\omega\neq 0$ and small $\delta\in (0,1)$, define
the interval
\begin{equation*}
I_{\lambda,\delta}(x,\omega)=\left\{y=x+s\omega:0<s\leq r,\
r^n=\min\left\{\frac{\delta^n}{b^n}|\nabla
f(x)\cdot\omega|^n,\frac{(1-\delta)^p|\nabla
f(x)\cdot\omega|^p}{\lambda^p||\omega||_K^{n+p}}\right\}\right\}.
\end{equation*}

\medskip

\noindent\textbf{Claim 1.} for any $x\in\mathbb R^n$ and $y\in
I_{\lambda,\delta}(x,\omega)$ such that $\nabla f(x)\cdot\omega\neq
0$, it holds that $(x,y)\in E_{\lambda,K}$.

\medskip

 \noindent Indeed, from the definition of $I_{\lambda,\delta}(x,\omega)$, we have
\begin{equation*}
br\leq \delta|\nabla f(x)\cdot\omega|\qquad\text{and}\qquad \lambda
||\omega||_K^{\frac np +1}r^{\frac np}\leq (1-\delta)|\nabla
f(x)\cdot\omega|,
\end{equation*}
and hence
\begin{equation*}
br+\lambda ||\omega||_K^{\frac np +1}r^{\frac np}\leq |\nabla
f(x)\cdot\omega|.
\end{equation*}
This together with \eqref{eq22} and the fact that $s=|x-y|\leq r$, we obtain
\begin{equation*}
|f(x)-f(y)|\geq |\nabla f(x)\cdot(x-y)|- b|x-y|^2\geq\lambda
|x-y|^{\frac np +1}||\omega||_K^{\frac np +1}=\lambda
||x-y||_K^{\frac np +1}.
\end{equation*}
Hence it follows from \eqref{eq6''} that $(x,y)\in E_{\lambda,K}$,
which proves Claim 1.

\medskip

By using Claim 1 and the polar coordinate, we obtain
\begin{align*}
\lambda^p\mathcal
L^{2n}\left(E_{\lambda,K}\right)&\geq\frac{\lambda^p}{n}\int_{\mathbb R^n}\int_{ S^{n-1}}\textbf{1}_{\nabla f(x)\cdot\omega\neq0}\cdot r^nd\omega dx\notag\\
&=\frac1n\int_{\mathbb R^n}\int_{ S^{n-1}}\textbf{1}_{\nabla
f(x)\cdot\omega\neq0}\cdot\min\left\{\frac{\lambda^p\delta^n}{b^n}|\nabla
f(x)\cdot\omega|^n,\frac{(1-\delta)^p|\nabla
f(x)\cdot\omega|^p}{||\omega||_K^{n+p}}\right\}d\omega dx.
\end{align*}
By the monotone convergence theorem, we further get
\begin{equation*}
\liminf_{\lambda\rightarrow\infty}\lambda^p\mathcal
L^{2n}\left(E_{\lambda,K}\right)\geq\frac{(1-\delta)^p}{n}\int_{\mathbb
R^n}\int_{ S^{n-1}}|\nabla
f(x)\cdot\omega|^p||\omega||_K^{-n-p}d\omega dx.
\end{equation*}
Since $\delta>0$ is arbitrary small, it follows from the polar
coordinate and \eqref{moment} that
\begin{align}
\liminf_{\lambda\rightarrow\infty}\lambda^p\mathcal
L^{2n}\left(E_{\lambda,K}\right)\geq&\frac{1}{n}\int_{\mathbb
R^n}\left(\int_{S^{n-1}}\left|\nabla f(x)\cdot
\omega\right|^p||\omega||_K^{-n-p}d\omega\right)dx\notag\\
=&\frac{2}{n}\int_{\mathbb R^n}||\nabla f(x)||_{Z_p^*K}^pdx.
\label{eq30}
\end{align}

\medskip

In view of \eqref{eq30}, to show \eqref{eq7''}, it suffices to prove that
\begin{equation*}
\limsup_{\lambda\rightarrow\infty}\lambda^p\mathcal
L^{2n}\left(E_{\lambda,K}\right)\leq\frac{2}{n}\int_{\mathbb R^n}||\nabla f(x)||_{Z_p^*K}^pdx.
\end{equation*}

For any $x\in \mathbb R^n$ and $\omega\in S^{n-1}$, define the
interval
\begin{equation*}
J_{\lambda}(x,\omega)=\left\{y=x+s\omega:0<s\leq R,
R^n=\frac{1}{\lambda^p ||\omega||_K^{n+p}}\left(|\nabla
f(x)\cdot\omega|+b\left(\frac{a}{\lambda ||\omega||_K^{\frac np
+1}}\right)^{\frac pn}\right)^p\right\}.
\end{equation*}

\bigskip

\noindent\textbf{Claim 2.}  if $(x,y)\in E_{\lambda,K}$ with
$\lambda>a ||\omega||_K^{-\frac np -1}$ and $\omega=(y-x)/|y-x|$,
then $y\in J_{\lambda}(x,\omega)$ and $\text{dist}(x,\text{supp}\
f)\leq 1$. Here $\text{dist}(x,\text{supp}\ f)$ is the Euclidean
distance between $x$ and the support of $f$.

\medskip

\noindent Indeed, it follows from \eqref{eq6''} that $(x,y)\in
E_{\lambda,K}$ implies
\begin{equation}\label{E}
|f(x)-f(y)|\geq \lambda ||x-y||_K^{\frac np +1}.
\end{equation}
From \eqref{eq22}, we get
\begin{equation}\label{E'}
|f(x)-f(y)|\leq |\nabla f(x)\cdot(x-y)|+b|x-y|^2.
\end{equation}
Hence by \eqref{E} and \eqref{E'},
\begin{equation}\label{eq32}
\lambda ||\omega||_K^{\frac np +1}s^{\frac np}\leq  |\nabla
f(x)\cdot\omega|+bs,
\end{equation}
where $s=|x-y|$ and $\omega=(y-x)/|y-x|$. Using \eqref{eq21} and
\eqref{E}, we further have
\begin{equation}\label{eq33}
\lambda ||\omega||_K^{\frac np +1}s^{\frac np}\leq  a.
\end{equation}
Substituting \eqref{eq33} into \eqref{eq32}, we obtain
\begin{equation*}
\lambda ||\omega||_K^{\frac np +1}s^{\frac np}\leq  |\nabla
f(x)\cdot\omega|+b\left(\frac{a}{\lambda ||\omega||_K^{\frac np
+1}}\right)^{\frac pn},
\end{equation*}
which implies $y\in J_{\lambda}(x,\omega)$.

It remains to
show that if $(x,y)\in E_{\lambda,K}$ with $\lambda>a
||\omega||_K^{-\frac np -1}$, then
\begin{equation}\label{eq35}
\text{dist}(x,\text{supp}\ f)\leq 1.
\end{equation}
Indeed, if $(x,y)\in E_{\lambda,K}$, then it follows from
\eqref{eq33} and the assumption $\lambda>a ||\omega||_K^{-\frac np
-1}$ that $s=|x-y|<1$. On the contrary, suppose
$\text{dist}(x,\text{supp}\ f)>1$. Then, from the above observation,
we must have $f(x)=f(y)=0$. Together with \eqref{E}, we get
$\lambda||\omega||_K^{\frac np +1}|x-y|^{\frac np
+1}\leq|f(x)-f(y)|=0$, which further implies that $x=y$. However, it
contradicts the fact that $(x,x)\notin E_{\lambda,K}$ for any $x\in
\mathbb R^n$, and hence \eqref{eq35} follows. This completes the
proof of Claim 2.

By Claim 2 and the polar coordinate, we obtain
\begin{align*}
\lambda^p\mathcal
L^{2n}\left(E_{\lambda,K}\right)&\leq\frac{\lambda^p}{n}\int_{\mathbb
R^n}\int_{S^{n-1}}\textbf{1}_{\text{dist}(x,\text{supp}\
f)\leq1}\cdot R^nd\omega dx\notag\\
&=\frac1n\int_{\mathbb R^n}\int_{
S^{n-1}}\textbf{1}_{\text{dist}(x,\text{supp}\
f)\leq1}\cdot\frac{1}{||\omega||_K^{n+p}}\left(|\nabla
f(x)\cdot\omega|+b\left(\frac{a}{\lambda ||\omega||_K^{\frac np
+1}}\right)^{\frac pn}\right)^pd\omega dx.
\end{align*}
By the dominated convergence theorem and \eqref{moment}, we further
get
\begin{equation*}
\limsup_{\lambda\rightarrow\infty}\lambda^p\mathcal
L^{2n}\left(E_{\lambda,K}\right)\leq\frac{1}{n}\int_{\mathbb
R^n}\left(\int_{S^{n-1}}\left|\nabla f(x)\cdot
\omega\right|^p||\omega||_K^{-n-p}d\omega\right)dx=\frac{2}{n}\int_{\mathbb
R^n}||\nabla f(x)||_{Z_p^*K}^pdx,
\end{equation*}
together with \eqref{eq30}, the desired result \eqref{eq7''}
follows.
\end{proof}

\subsection*{Proof of Theorem \ref{mainthm}(b)}
Inspired by the technique developed in \cite{GY}, the proof of
Theorem \ref{mainthm}(b) can be divided by the following two lemmas.
The second inequality of \eqref{eqeq5'} follows from Lemma \ref{L3}.
The limiting behavior \eqref{eqeq7'} will be established in Lemma
\ref{L4}.

\begin{lemma}\label{L3}For $f\in L^p(\mathbb R^n)$,
\begin{equation}\label{eqeqeq1}
\left[\frac{f(x)-f(y)}{||x-y||_K^{\frac np}}\right]^p_{M^p(\mathbb
R^n\times\mathbb R^n,\mathcal L^{2n})} \leq 2^{p+1}|K|||f||_{L^p(\mathbb R^n)}^p.
\end{equation}
\end{lemma}
\begin{proof}
Let $f\in L^p(\mathbb R^n)$ and $\lambda>0$, denote
\begin{equation}\label{EE}
\widetilde E_{\lambda,K}=\left\{(x,y)\in\mathbb R^n\times\mathbb
R^n:\ x\neq
y,\frac{|f(x)-f(y)|}{||x-y||_K^{\frac{n}{p}}}\geq\lambda\right\}.
\end{equation}
Clearly,
\begin{equation*}
\widetilde E_{\lambda,K}\subseteq\left\{(x,y)\in\mathbb
R^n\times\mathbb R^n: \frac{|f(x)|}{||x-y||_K^{\frac
np}}\geq\frac{\lambda}2\right\}\bigcup\left\{(x,y)\in\mathbb
R^n\times\mathbb R^n: \frac{|f(y)|}{||x-y||_K^{\frac
np}}\geq\frac{\lambda}2\right\},
\end{equation*}
and by symmetry, the two sets on the RHS have the same $2n$-Lebesgue
measure. On the other hand, the change of variable $z=y-x$ and the
definition of the Minkowski functional \eqref{Min}, we have, for any
given $x\in \mathbb R^n$,
\begin{equation}\label{com}
\int_{\mathbb R^n}\mathbf{1}_{\left\{y:\
||y-x||_K\leq\left(\frac{2|f(x)|}{\lambda}\right)^{\frac
pn}\right\}}dy =\int_{\mathbb R^n}\mathbf{1}_{\left\{z+x:\
||z||_{\left(\frac{2|f(x)|}{\lambda}\right)^{\frac pn}K}\leq
1\right\}}dz=\left(\frac{2|f(x)|}{\lambda}\right)^{p}|K|.
\end{equation}
Therefore,
\begin{align}
\mathcal L^{2n}(\widetilde E_{\lambda,K})&\leq 2\int_{\mathbb R^n}\int_{\mathbb R^n}\mathbf{1}_{\left\{y:\ ||y-x||_K\leq\left(\frac{2|f(x)|}{\lambda}\right)^{\frac pn}\right\}}dydx\notag\\
&=2\int_{\mathbb R^n}\left(\frac{2|f(x)|}{\lambda}\right)^{p}|K|dx\notag\\
&=\frac{2^{p+1}|K|}{\lambda^p}||f||_{L^p(\mathbb
R^n)}^p.\label{eqeqeq2}
\end{align}
\noindent Multiplying by $\lambda^p$ on both sides of
\eqref{eqeqeq2} and taking supremum on $\lambda$, the desired
inequality \eqref{eqeqeq1} follows from the definition of the
weak $L^p$ quasinorm \eqref{Mar}.
\end{proof}

\begin{lemma}\label{L4}
For $f\in L^p(\mathbb R^n)$, if we denote
\begin{equation*}\label{eqeqeq3}
\widetilde E_{\lambda,K}=\left\{(x,y)\in\mathbb R^n\times\mathbb R^n:\ x\neq
y,\frac{|f(x)-f(y)|}{||x-y||_K^{\frac{n}{p}}}\geq\lambda\right\},
\end{equation*}
then
\begin{equation}\label{eqeqeq4}
\lim_{\lambda\rightarrow0^+}\lambda^p\mathcal
L^{2n}(\widetilde E_{\lambda,K})=2|K|||f||_{L^p}^p.
\end{equation}
\end{lemma}
\begin{proof}
We first deal with the case that $f$ is compactly supported. Then we
prove the lemma for general case that $f\in L^p(\mathbb R^n)$ by
using suitable truncations of $f$.

\noindent\textbf{Case 1. $f$ is compactly supported.} For $\lambda >
0$, let
\begin{equation*}
H_{\lambda,K}^+ =\left\{(x,y) \in \mathbb R^n \times \mathbb R^n
\colon
||y||_K>||x||_K,\frac{|f(x)-f(y)|}{||x-y||_K^{\frac{n}{p}}}\geq\lambda\right\},
\end{equation*}
\begin{equation*}
H_{\lambda,K}^- =\left\{(x,y) \in \mathbb R^n \times \mathbb R^n
\colon
||y||_K<||x||_K,\frac{|f(x)-f(y)|}{||x-y||_K^{\frac{n}{p}}}\geq\lambda\right\},
\end{equation*}
and
\begin{equation*}
H_{\lambda,K} =\left\{(x,y) \in \mathbb R^n \times \mathbb R^n
\colon
||y||_K=||x||_K,\frac{|f(x)-f(y)|}{||x-y||_K^{\frac{n}{p}}}\geq\lambda\right\}.
\end{equation*}
Note that it follows from \eqref{Min} that $\mathcal
L^{2n}(H_{\lambda,K} )=0$. Due to the fact that $K$ is
origin-symmetric, we further have $\mathcal L^{2n}(H_{\lambda,K}^+
)=\mathcal L^{2n}(H_{\lambda,K}^-)$. Recall the definition of
$\widetilde E_{\lambda,K}$ in \eqref{EE}, clearly we have
\begin{equation}\label{eq''}
\mathcal L^{2n}(\widetilde E_{\lambda,K})=2\mathcal
L^{2n}(H_{\lambda,K}^+).
\end{equation}

Since $f$ is compactly supported, we may assume
\begin{equation*}%
\text{supp} \, f \subseteq rK
\end{equation*}%
for some $r > 0$. Observe that if $(x,y) \in H_{\lambda,K}^+$, then
we must have $x \in rK$. Otherwise, it means that $\|y\|_K>\|x\|_K>
r$. Thus, both $x,y$ are outside $rK$. Now, our assumption on
$\text{supp} \, f$ yields that $f(x)=f(y)=0$, and hence $(x,y)
\notin \widetilde E_{\lambda,K}$, contradicting the fact that $(x,y)
\in H_{\lambda,K}^+\subseteq \widetilde E_{\lambda,K}$. For any
given $x \in rK$, let
\begin{equation*}%
H_{\lambda,K,x}^+ :=\left\{y \in \mathbb R^n \colon ||y||_K>||x||_K,
\,\frac{|f(x)-f(y)|}{||x-y||_K^{\frac{n}{p}}}\geq\lambda\right\}
\end{equation*}%
and
\begin{align*}%
H_{\lambda,K,x,r}^+:&= \left\{y \in \mathbb R^n \colon ||y||_K > r,
\,
\frac{|f(x)-f(y)|}{||x-y||_K^{\frac{n}{p}}}\geq\lambda\right\}\notag\\&=\left\{y
\in \mathbb R^n \colon ||y||_K > r, \,
\frac{|f(x)|}{||x-y||_K^{\frac{n}{p}}}\geq\lambda\right\}.
\end{align*}%
Here we use the fact that for $||y||_K > r$, $f(y) = 0$. Therefore,
\begin{equation} \label{eq:Hinclusion}
H_{\lambda,K,x,r}^+ = H_{\lambda,K,x}^+ \setminus rK \subseteq
H_{\lambda,K,x}^+  \subseteq H_{\lambda,K,x,r}^+  \cup rK.
\end{equation}
Like the computation \eqref{com}, it follows from  the first
inclusion in \eqref{eq:Hinclusion} that
\begin{equation}\label{eqeqeq3}
\mathcal{L}^n(H_{\lambda,K,x}^+) \geq
\mathcal{L}^n(H_{\lambda,K,x,r}^+) \geq |K|\left(
\frac{|f(x)|^p}{\lambda^p} - r^n\right).
\end{equation}
Also, it follows from the second inclusion in \eqref{eq:Hinclusion}
that
\begin{equation}\label{eqeqeq3'}
\mathcal{L}^n(H_{\lambda,K,x}^+) \leq |K|\left(
\frac{|f(x)|^p}{\lambda^p} + r^n\right).
\end{equation}
Note that Fubini's theorem implies
\begin{equation} \label{eq:Fub}
\mathcal{L}^{2n}(H_{\lambda,K}^+) = \int_{rK}
\mathcal{L}^n(H_{\lambda,K,x}^+) dx.
\end{equation}
Now integrating \eqref{eqeqeq3} and \eqref{eqeqeq3'} over $x \in rK$,
formula \eqref{eq:Fub} yields
\begin{equation}\label{m1}%
\frac{|K|}{\lambda^p} \|f\|^p_{L^p(\mathbb R^n)} - |K|^2 r^{2n} \leq
\mathcal{L}^{2n}(H_{\lambda,K}^+) \leq \frac{|K|}{\lambda^p}
\|f\|^p_{L^p(\mathbb R^n)} + |K|^2 r^{2n}.
\end{equation}%
Multiplying both sides by $\lambda^p$ and letting $\lambda
\rightarrow 0^+$, we obtain
\begin{equation}\label{eqsemifinal}
\lim_{\lambda \rightarrow 0^+} \lambda^p
\mathcal{L}^{2n}(H_{\lambda,K}^+) = |K| \|f\|^p_{L^p(\mathbb R^n)}.
\end{equation}
Finally, it follows from \eqref{eq''} and \eqref{eqsemifinal} that
\begin{equation*}%
\lim_{\lambda\rightarrow0^+}\lambda^p\mathcal L^{2n}(\widetilde
E_{\lambda,K})= 2|K| \|f\|^p_{L^p(\mathbb R^n)},
\end{equation*}%
which completes the proof for case 1.

\bigskip

\noindent\textbf{Case 2. $f \in L^p(\mathbb R^n)$ is not necessarily
compactly supported.} Let $f_r = f\cdot \mathbf{1}_{rK}$ be the
truncation of $f$ in $rK$ for some $r>0$ and let $g_r=f-f_r$. Since $f \in
L^p(\mathbb R^n)$ for some $1 \leq p < \infty$, we must have
$\|g_r\|_{L^p(\mathbb R^n)}\rightarrow0$ as $r\rightarrow \infty$.

Since $f = f_r + g_r$, the triangle inequality yields that
\begin{equation*}
\frac{|f(x)-f(y)|}{||x-y||_K^{\frac{n}{p}}}\leq
\frac{|f_r(x)-f_r(y)|}{||x-y||_K^{\frac{n}{p}}}+\frac{|g_r(x)-g_r(y)|}{||x-y||_K^{\frac{n}{p}}}.
\end{equation*}
Hence, for any $\sigma\in(0,1)$, we have
\begin{equation*} \label{eq:Edec1}
\widetilde E_{\lambda,K} = \left\{(x,y) \in \mathbb R^n \times
\mathbb R^n \colon
\frac{|f(x)-f(y)|}{||x-y||_K^{\frac{n}{p}}}\geq\lambda\right\}
\subseteq A_f \cup A_g,
\end{equation*}
where
\begin{equation*}
A_f := \left\{(x,y) \in \mathbb R^n \times \mathbb R^n \colon
\frac{|f_r(x)-f_r(y)|}{||x-y||_K^{\frac{n}{p}}}\geq\lambda(1-\sigma)\right\}
\end{equation*}
and
\begin{equation} \label{eq:A2}
A_g := \left\{(x,y) \in \mathbb R^n \times \mathbb R^n \colon
\frac{|g_r(x)-g_r(y)|}{||x-y||_K^{\frac{n}{p}}}\geq
{\lambda\sigma}\right\}.
\end{equation}
Moreover,
\begin{equation}\label{eqeqeq17'}
\mathcal L^{2n}(\widetilde E_{\lambda,K})\leq \mathcal
L^{2n}(A_f)+\mathcal L^{2n}(A_g).
\end{equation}
Since $f_r$ is compactly supported in $rK$, replacing $\lambda$ by
$\lambda(1-\sigma)$ in the second inequality of \eqref{m1} and
\eqref{eq''}, we get
\begin{equation}\label{eqeqeq11}
\mathcal L^{2n}(A_f)\leq
\frac{2|K|}{\lambda^p(1-\sigma)^p}\|f_r\|_{L^p(\mathbb
R^n)}^p+2|K|^2 r^{2n}.
\end{equation}
For $A_g$, replacing $\lambda$ by $\lambda\sigma$ in \eqref{eqeqeq2}
for $g_r$, we have
\begin{equation}
\mathcal
L^{2n}(A_g)\leq\frac{2^{p+1}|K|}{\lambda^{p}\sigma^{p}}\|g_r\|^p_{L^p(\mathbb
R^n)}.\label{eqeqeq9}
\end{equation}
Together with \eqref{eqeqeq17'}, \eqref{eqeqeq11} and
\eqref{eqeqeq9}, multiplying by $\lambda^p$, we obtain
\begin{equation*}\label{eqeqeq15}
\lambda^p\mathcal L^{2n}(\widetilde  E_{\lambda,K})\leq
\frac{2|K|}{(1-\sigma)^p}\|f_r\|_{L^p(\mathbb
R^n)}^p+2\lambda^p|K|^2
r^{2n}+\frac{2^{p+1}|K|}{\sigma^p}\|g_r\|_{L^p(\mathbb R^n)}^p.
\end{equation*}
Now first let $\lambda\rightarrow 0^+$, then let
$r\rightarrow\infty$ and finally let $\sigma\rightarrow 0^+$, then
the facts
\begin{equation}\label{R}%
\lim_{r \rightarrow \infty} \|f_r\|_{L^p(\mathbb R^n)} =
\|f\|_{L^p(\mathbb R^n)} \quad \text{and} \quad \lim_{r \rightarrow
\infty} \|g_r\|_{L^p(\mathbb R^n)} = 0
\end{equation}%
yield that
\begin{equation}\label{eqeqeq16}
\limsup_{\lambda\rightarrow0^+}\lambda^p\mathcal L^{2n}(\widetilde
E_{\lambda,K})\leq 2|K|\|f\|_{L^p(\mathbb R^n)}^p.
\end{equation}

Similarly, the triangle inequality also yields that
\begin{equation*}
\frac{|f(x)-f(y)|}{||x-y||_K^{\frac{n}{p}}}\geq
\frac{|f_r(x)-f_r(y)|}{||x-y||_K^{\frac{n}{p}}}-\frac{|g_r(x)-g_r(y)|}{||x-y||_K^{\frac{n}{p}}}.
\end{equation*}
Hence, for any $\sigma > 0$, we have
\begin{equation*}%
\widetilde  E_{\lambda,K}=\left\{(x,y) \in \mathbb R^n \times
\mathbb R^n \colon
\frac{|f(x)-f(y)|}{||x-y||_K^{\frac{n}{p}}}\geq\lambda\right\}\supseteq
\overline{A}_f\setminus A_g
\end{equation*}%
where
\begin{equation*}
\overline{A}_f := \left\{(x,y) \in \mathbb R^n \times \mathbb R^n
\colon
\frac{|f_r(x)-f_r(y)|}{||x-y||_K^{\frac{n}{p}}}\geq\lambda(1+\sigma)\right\}
\end{equation*}
and $A_g$ is defined in \eqref{eq:A2}. Hence
\begin{equation}\label{eqeqeq17}
\mathcal L^{2n}(\widetilde  E_{\lambda,K})\geq \mathcal
L^{2n}(\overline{A}_f)-\mathcal L^{2n}(A_g).
\end{equation}
Since $f_r$ is compactly supported in $rK$, replacing $\lambda$ by
$\lambda(1+\sigma)$ in the first inequality of \eqref{m1} and
\eqref{eq''}, we get
\begin{equation}\label{eqeqeq18}
\mathcal L^{2n}(\overline{A}_f)\geq
\frac{2|K|}{\lambda^p(1+\sigma)^p}\|f_r\|_{L^p(\mathbb
R^n)}^p-2|K|^2 r^{2n}.
\end{equation}
Together with \eqref{eqeqeq9}, \eqref{eqeqeq17} and
\eqref{eqeqeq18}, multiplying by $\lambda^p$, we obtain
\begin{equation*}\label{eqeqeq19}
\lambda^p\mathcal L^{2n}(\widetilde E_{\lambda,K})\geq
\frac{2|K|}{(1+\sigma)^p}\|f_r\|_{L^p(\mathbb
R^n)}^p-2\lambda^p|K|^2
r^{2n}-\frac{2^{p+1}|K|}{\sigma^p}\|g_r\|_{L^p(\mathbb R^n)}^p.
\end{equation*}
Now first let $\lambda\rightarrow 0^+$, then let $r
\rightarrow\infty$ and finally let $\sigma \rightarrow 0^+$,
together with \eqref{R}, we obtain
\begin{equation}\label{eqeqeq20}
\liminf_{\lambda\rightarrow0^+}\lambda^p\mathcal L^{2n}(\widetilde
E_{\lambda,K})\geq 2|K|\|f\|_{L^p(\mathbb R^n)}^p.
\end{equation}
Consequently, \eqref{eqeqeq4} follows from
\eqref{eqeqeq16} and \eqref{eqeqeq20}.

\end{proof}

\bigskip

\bibliographystyle{siam}

\bigskip
\bigskip

\end{document}